\pdfoutput=1
\RequirePackage{ifpdf}
\ifpdf 
\documentclass[pdftex]{sigma}
\else
\documentclass{sigma}
\fi

\usepackage[all]{xy}

\newcommand{\N}{\mathbb N}
\newcommand{\Z}{\mathbb Z}
\newcommand{\Q}{\mathbb Q}
\newcommand{\Zbar}{\overline{\Z}}
\newcommand{\Qbar}{\overline{\Q}}
\newcommand{\F}{\mathbb F}

\newcommand{\Zmod}[1]{\overline{\Z/{#1}\Z}}

\newcommand{\legendre}[2]{\left(\frac{#1}{#2}\right)}

\newcommand{\abcd}[4]{\left(\begin{smallmatrix}#1&#2\\#3&#4\end{smallmatrix}\right)}

\newtheorem{que}{Question}

\DeclareMathOperator{\GL}{GL} \DeclareMathOperator{\SL}{SL}  
\DeclareMathOperator{\Gal}{Gal}   
  \DeclareMathOperator{\Tr}{Tr} 
\DeclareMathOperator{\Frob}{Frob}

\begin{document}
\allowdisplaybreaks

\newcommand{\arXivNumber}{1802.04976}

\renewcommand{\thefootnote}{}

\renewcommand{\PaperNumber}{057}

\FirstPageHeading

\ShortArticleName{A Peculiar Eigenform Modulo 4}

\ArticleName{Dihedral Group, 4-Torsion on an Elliptic Curve,\\ and a Peculiar Eigenform Modulo 4\footnote{This paper is a~contribution to the Special Issue on Modular Forms and String Theory in honor of Noriko Yui. The full collection is available at \href{http://www.emis.de/journals/SIGMA/modular-forms.html}{http://www.emis.de/journals/SIGMA/modular-forms.html}}}

\Author{Ian KIMING~$^\dag$ and Nadim RUSTOM~$^\ddag$}

\AuthorNameForHeading{I.~Kiming and N.~Rustom}

\Address{$^\dag$~Department of Mathematical Sciences, University of Copenhagen,\\
\hphantom{$^\dag$}~Universitetsparken 5, DK-2100 Copenhagen {\O}, Denmark}
\EmailD{\href{mailto:kiming@math.ku.dk}{kiming@math.ku.dk}}
\URLaddressD{\url{http://web.math.ku.dk/~kiming/}}

\Address{$^\ddag$~Department of Mathematics, Ko\c{c} University, Rumelifeneri Yolu,\\
\hphantom{$^\ddag$}~34450, Sariyer, Istanbul, Turkey}
\EmailD{\href{mailto:restom.nadim@gmail.com}{restom.nadim@gmail.com}}

\ArticleDates{Received February 28, 2018, in final form June 04, 2018; Published online June 13, 2018}

\Abstract{We work out a non-trivial example of lifting a so-called weak eigenform to a true, characteristic $0$ eigenform. The weak eigenform is closely related to Ramanujan's tau function whereas the characteristic $0$ eigenform is attached to an elliptic curve defined over~$\Q$. We produce the lift by showing that the coefficients of the initial, weak eigenform (almost all) occur as traces of Frobenii in the Galois representation on the $4$-torsion of the elliptic curve. The example is remarkable as the initial form is known not to be liftable to any characteristic $0$ eigenform of level~$1$. We use this example as illustrating certain questions that have arisen lately in the theory of modular forms modulo prime powers. We give a~brief survey of those questions.}


\Classification{11F33; 11F80}

\renewcommand{\thefootnote}{\arabic{footnote}}
\setcounter{footnote}{0}

\section{Introduction}\label{section:intro} In what follows, $\ell$ will always denote a prime number, and we will denote by $G_{\Q}$ the absolute Galois group of $\Q$:
\begin{gather*}
G_{\Q} := \Gal\big(\Qbar/\Q\big) .
\end{gather*}

We will be considering the modular form
\begin{gather*}
f:=E_4^6\Delta + 2\Delta^3
\end{gather*}
on $\SL_2(\Z)$ of weight $36$. Here $E_4$ is the classical Eisenstein series
\begin{gather*}
E_4 := 1 + 240 \sum_{n=1}^{\infty} \bigg( \sum_{d\mid n} d^3 \bigg) q^n
\end{gather*}
($q := e^{2\pi i z}$), and $\Delta$ the unique normalized eigenform on $\SL_2(\Z)$ of weight $12$
\begin{gather*}
\Delta = \sum_{n=1}^{\infty} \tau(n) q^n
\end{gather*}
with $\tau$ the Ramanujan tau function. Thus, the $q$-expansion of $f$ begins like this
\begin{gather*}
f = q + 1416 q^2 + 842654 q^3 + 271386544 q^4 + 50558981478 q^5 \\
\hphantom{f=}{} + 5356057726176 q^6 + 290719505955308 q^7 + \cdots.
\end{gather*}

For any modular form $h$ (of some weight on some $\Gamma_1(N)$) we will denote by $a_n(h)$ the coefficient of $q^n$ in the $q$-expansion of $h$. Thus, $a_n(\Delta) = \tau(n)$ for instance.

The result of this paper is a non-trivial observation concerning the reduction of the form $f$ modulo $4$. Let us first explain the result, and then return in the next section to a discussion of background, general motivations, why the form $f$ is particularly interesting, as well as some of the questions that form the motivating background for studying the particular form $f$. It is the study of those questions that is our main purpose.

Before formulating our result, let us first explain what we mean by ``reduction of'' and ``congruences between'' modular forms. Assuming for the moment that the forms in question have integral $q$-expansions, the reduction modulo $m$ of a form is by definition the formal power series in $(\Z/m\Z)[[q]]$ that is the result of reducing the $q$-expansion of the form term-wise modulo~$m$. Accordingly, that two forms are congruent modulo $m$ means that their $q$-expansions agree term-wise modulo~$m$.

The form $f$ modulo $4$ is an eigenform for all Hecke operators meaning that we have
\begin{gather*}
T_n f \equiv a_n(f) f \pmod{4}
\end{gather*}
for all $n\in\N$ where $T_n$ is the $n$th Hecke operator at level $1$. For the short proof of this, see Proposition \ref{prop:f_is_weak} in the next section. The form $(f\pmod{4})$ is an example of what we will call a~weak eigenform (in this case modulo~$4$), cf.\ Definition \ref{def:weak_strong_eigenform} below.

The form mod $4$ does not lift to an eigenform of characteristic $0$ on $\SL_2(\Z)$: by results of Hatada, any eigenform on $\SL_2(\Z)$ is congruent modulo $4$ to $\Delta$ (this follows from Theorems~3 and~4 of~\cite{hatada1977b} together with the fact that the coefficients $\tau(\ell)$ of $\Delta$ satisfy $\tau(\ell) \equiv 1 + \ell \pmod{4}$ for odd primes $\ell$, and that $\tau(2) \equiv 0 \pmod{4}$ (see, e.g.,~\cite{swinnerton-dyer})).

Nevertheless, the result of this paper shows that $f$ does in fact lift to a classical newform, albeit at level $128$:

\begin{theorem}\label{thm:main} With $f:=E_4^6\Delta + 2\Delta^3$ and $E$ the elliptic curve
\begin{gather*}
E \colon \ y^2 = x^3+x^2+x+1 ,
\end{gather*}
we have that the $2$-adic Galois representation $\rho_2$ attached to $E$ is unramified outside $2$ and that
\begin{gather*}
\Tr \rho_2(\Frob_\ell) \equiv a_\ell(f) \pmod{4}
\end{gather*}
for odd primes $\ell$.

In fact,
\begin{gather*}
f\equiv g \pmod{4}
\end{gather*}
with $g$ the cusp form of weight $2$ and level $128$ attached to $E$.
\end{theorem}

Of course, a possible, quick and brutal proof of Theorem \ref{thm:main} would be to just verify that $a_n(f) \equiv a_n(g) \pmod{4}$ for sufficiently many $n$, cf.\ the ``generalized Sturm bounds'' of~\cite{ckr}. But such a proof would not reveal how the curve $E$, and consequently the form $g$, was found. The proof that we give below in Section~\ref{section:main} gives more information concerning that point, as well as information such as a formula for $a_\ell(f) \pmod{4}$ and the structure of $\rho_2 \pmod{4}$.

\section{Background, motivations, questions}\label{section:background} The background for our study of the form $f$ is the theory of modular forms modulo prime powers that has attracted some attention in recent years, cf.\ for instance the papers \cite{ck,ckr, ckw,krw,nadim_filtrations,nadim_level1,XaGa,panos_thesis,panos_gabor_survey}. Before reviewing some of the questions that have arisen recently, let us set up some notation.

Let $m,N\in\N$ and let $p$ be a prime number not dividing $N$. We fix algebraic closures $\Qbar$, $\Qbar_p$ of $\Q$ and $\Q_p$, respectively, as well as an embedding $\Qbar \hookrightarrow \Qbar_p$. In $\Qbar_p$ we have the subring $\Zbar_p$, the elements integral over $\Z_p$. Let $v_p$ be the normalized valuation on $\Qbar_p$.

The ring
\begin{gather*}
\Zmod{p^m} = \Zbar_p/\big\{x \in \Zbar_p \,|\, v_p(x)>m-1\big\},
\end{gather*}
introduced in \cite{XaGa}, utilized in~\cite{ckw}, is the natural target for ``reductions modulo $p^m$'' when the things that are being reduced are in $\Zbar_p$, but not necessarily in $\Z_p$; cf.\ the discussions in the papers just cited.

Thus, to say that two modular forms with $q$-expansions in $\Zbar_p$ are ``congruent modulo $p^m$'' means that their $q$-expansions term-wise have the same reduction in $\Zmod{p^m}$.

Considering modular forms $h$ on $\Gamma_1(N)$ and with coefficients in $\Zbar_p$ we will now consider various ways in which the reduction $(h \pmod{p^m}$ of $h$ modulo $p^m$ will be called an ``eigenform'' for the full Hecke algebra. For this, it will be convenient for us, as well as sufficient for our purposes, to restrict our attention to normalized cusp forms, i.e., cusp forms $h$ with $a_1(h) = 1$.

\begin{definition}\label{def:weak_strong_eigenform} Let $h$ be a normalized cusp form on $\Gamma_1(N)$, some weight, and coefficients in $\Zbar_p$. We say that the reduction $(h \pmod{p^m}$ of $h$ modulo $p^m$ is a weak eigenform if
\begin{gather*}
(T_n h \pmod{p^m}) = a_n(h)\cdot (h \pmod{p^m})
\end{gather*}
for all $n$ with $T_n$ the $n$th Hecke operator at level $N$ and the weight in question.

We say that $(h \pmod{p^m}$ is a strong eigenform (relative to the fixed level $N$) if there exists a classical, normalized eigenform $h_1$ on $\Gamma_1(N)$ and some weight such that
\begin{gather*}
(h \pmod{p^m}) = (h_1 \pmod{p^m}) .
\end{gather*}
\end{definition}

There is a third, weaker sense of the notion of an eigenform mod $p^m$, namely the notion of a ``dc-weak eigenform''; ``dc'' stands here for ``divided congruence'', and it is a notion of being an eigenform that involves the Hecke algebras simultaneously at various weights, but still at the fixed level $N$. It is not necessary for the purposes of this paper to give the precise definition which can be found in \cite[Section 1]{ckw}. At a fixed level $N$, we have
\begin{gather*}
\{ \mbox{strong eigenforms mod $p^m$}\} \subseteq \{\mbox{weak eigenforms mod $p^m$}\} \\
\hphantom{\{ \mbox{wtrong eigenforms mod $p^m$}\}}{} \subseteq \{ \mbox{dc-weak eigenforms mod $p^m$}\},
\end{gather*}
and the inclusions can in fact be strict: the first examples of dc-weak eigenforms that are not weak can be found in \cite{nadim_level1} (Theorem 5.4 and the consequences of it; the simplest example is the form $\Delta + 2d\Delta$ modulo $4$ where $d:=(E_4 -1)/16$.)

As an example of a weak eigenform that is not strong, as noted in the introduction above, one can consider the principal object of this paper, i.e., the modular form $f:=E_4^6\Delta + 2\Delta^3$ considered modulo $4$ (and so we have $N=1$, $p=2$, $m=2$): we already noted above that, due to a classical result of Hatada, $(f\pmod{4})$ can not be the reduction of any normalized eigenform of any weight on $\SL_2(\Z)$, i.e., $(f\pmod{4})$ is not a strong eigenform in the sense of Definition~\ref{def:weak_strong_eigenform}. But we have the following.

\begin{proposition}\label{prop:f_is_weak} The form $(f\pmod{4})$ is a weak eigenform modulo $4$.
\end{proposition}

\begin{proof} The form $f$ lives in the space $S_{36}(1,\Z)$ of cusp forms on $\SL_2(\Z)$ of weight $36$ and coefficients in $\Z$. From classical theory we know that this space is a free $\Z$-module of rank $3$, generated by $(E_4)^6 \Delta$, $(E_4)^3 \Delta^2$, $\Delta^3$. This means that any form $h$ in this space is determined by its first three Fourier coefficients, $a_1(h)$, $a_2(h)$, $a_3(h)$ and also that its reduction modulo $4$ is determined by the reduction of these first three coefficients. I.e., two forms in the space are congruent modulo $4$ if and only if their first three coefficients are congruent modulo~$4$.

A further implication of the above is that the full Hecke algebra acting on this space is generated by $T_1$ (which is the identity) together with $T_2$ and $T_3$.

It follows from these facts, as well as the facts that we have $a_2(f) \equiv 0 \pmod{4}$, $a_3(f) \equiv 2 \pmod{4}$, that the reduction $(f\pmod{4})$ mod $4$ of the normalized form $f$ is a~weak eigenform if and only
\begin{gather*}
T_2 f \equiv a_2 f \equiv 0 \pmod{4}
\end{gather*}
and
\begin{gather*}
T_3 f \equiv a_3 f \equiv 2f \pmod{4} ,
\end{gather*}
and again these congruences can be proved by just verifying them for the first three coefficients.

Now, from the $q$-expansion of $f$ one computes
\begin{gather*}
T_2 f = 1416 q + 34631124912 q^2 + 5356057726176 q^3 + O\big(q^4\big) \equiv 0 + O\big(q^4\big) \pmod{4}
\end{gather*}
and
\begin{gather*}
 T_3 f - 2f = 842652 q + 5356057723344 q^2 + 113674493459566148 q^3 + O\big(q^4\big) \\
\hphantom{T_3 f - 2f}{} \equiv 0 + O\big(q^4\big) \pmod{4} ,
\end{gather*}
and thus the claim follows.
\end{proof}

Part of the reason for considering these different notions of ``eigenform modulo $p^m$'' is that one can attach Galois representations in the usual sense to such eigenforms, at least when the Galois representation attached to the residual (mod $p$) form is absolutely irreducible; see \cite[Theorem~2]{ckw}. Another reason is that, even if one is only interested in classical, characteristic~$0$ eigenforms, when level-lowering modulo $p^m$ is considered, the dc-weak eigenforms seem to enter as the natural framework (cf.\ \cite[Proposition~19]{ckw}.) Thirdly, it is possible under certain technical hypotheses to prove that a continuous representation of $G_{\Q}$ into $\GL_2\big(\Zmod{p^m}\big)$ is ``dc-weak modular'' in the sense that it is attached to a dc-weak eigenform modulo~$p^m$; see \cite[Theorem~3]{panos_gabor_survey}.

We shall now address the question of why one should be interested in the topic of modular forms modulo prime powers in the first place. We offer three reasons below.

First, even if we are only interested in classical, characteristic $0$ eigenforms, we find it a very natural question to understand congruences between them. Some fundamental information, by now classical, is available, notably papers by Katz~\cite{katz_higher_congruences} and Swinnerton--Dyer~\cite{swinnerton-dyer}. However, these do not answer all natural questions, such as for instance the following.

\begin{que}\label{q:finiteness} Fix a level $N$ and $m\in\N$. Consider the set of all eigenforms on $\Gamma_1(N)$ and all weights. Is the set of reductions modulo $p^m$ of these a finite set?
\end{que}

Notice that the answer to this question is affirmative when $m=1$, cf.\ the classical paper~\cite{jochnowitz_finiteness} by Jochnowitz. An affirmative answer for all $m$ was conjectured in \cite[Conjecture 1]{krw} that also gave some (weak) evidence in favor of it. Further evidence (for $N=1$) has been obtained in~\cite{nadim_level1}.

The results of the paper \cite{krw} suggested that a key to approach Question~\ref{q:finiteness} is to understand which weak eigenforms are the reductions of classical, characteristic $0$ eigenforms. A simple construction by Calegari and Emerton, see Section~3 of~\cite{calegari_emerton_large_index}, shows that it may happen that one has infinitely many weak eigenforms modulo~$p^2$ at some fixed level and weight; obviously, these can not all lift to characteristic~$0$ eigenforms {\it at the same fixed level and weight}. It seems to be not immediately clear whether we can lift if we are allowed to change the level and weight of the lifting form.

The origin of the present paper was to investigate the situation through a concrete, non-trivial example. Hence we picked the form $f \pmod{4}$ as a candidate for non-liftability since we knew that it does not lift at level $1$ and any weight. But then, after some lengthy analysis, we found it to lift anyway, but at level $128$ and weight $2$, hence Theorem \ref{thm:main}.

Our main purpose is thus to draw attention to Question \ref{q:finiteness} and in particular the following questions.

\begin{que}\label{q:strong_lifts} Does there exist a weak eigenform modulo $p^m$ $($some level $N$, some weight$)$ that is not the reduction modulo $p^m$ of a characteristic $0$ eigenform of some level $M$ and some weight?
\end{que}

Because of the more complicated behavior of ($p$-adic Galois representations attached to) eigenforms locally at the primes dividing $Np$ where $N$ is the level, one might want to weaken the notion of liftability to characteristic $0$ a bit and ask the following question. Thus, an affirmative answer to the following question would be a stronger statement than an affirmative answer to Question \ref{q:strong_lifts}.

\begin{que}\label{q:strong_lifts_outside_S} Does there exist a weak eigenform $h_0$ modulo $p^m$ $($some level $N$, some weight$)$ such that for any characteristic $0$ eigenform $h$ $($of some level $M$ and some weight$)$, we do not have
\begin{gather*}
a_\ell(h_0) \equiv a_\ell(h) \pmod{p^m}
\end{gather*}
for almost all primes $\ell$?
\end{que}

\begin{que}\label{q:strong_lifts_outside_S_level_Npt} Same question as in Question~{\rm \ref{q:strong_lifts_outside_S}}, but restrict the ``lifting level'' $M$ to numbers of the form~$Np^t$ for some~$t$.
\end{que}

We honestly do not have conjectures about the answers to the last three questions above either way, but find them interesting. The Calegari--Emerton construction mentioned above certainly pulls in the direction of conjecturing affirmative answers. On the other hand, there are results such as \cite{camporino_pacetti} that imply strong modularity of certain weak eigenforms under some technical restrictions.
\medskip

There are various other valid motivations for studying modular forms modular primes powers apart from the above. For instance, P.~Tsaknias and G.~Wiese argue in the introduction to~\cite{panos_gabor_survey} that a better understanding of this area might contribute to our algorithmic handling of the $p$-adic Galois representations attached to classical, characteristic~$0$ eigenforms.

Also, quite independently of the connection to Galois representations, the study of congruences between modular forms is natural if we consider the fact that coefficients of modular forms can be carriers of number theoretic or combinatorial information (representation numbers for positive definite, integral quadratic forms, partition problems, e.g.,~\cite{kiming_barcore_asymp,kiming_barcore,kiming_olsson}). Studying congruences between such numbers is a natural part of number theory.

\section{Proof of Theorem~\ref{thm:main}}\label{section:main}

\subsection[The coefficients of the form $f$ mod~4]{The coefficients of the form $\boldsymbol{f}$ mod~4}\label{section:coeffs_of_f} As in the introduction and everywhere in this paper, we put $f:=E_4^6\Delta + 2\Delta^3$.

For primes $\ell$ we have the following formula for $a_{\ell}(f)$ modulo $4$.

\begin{proposition}\label{prop:f_mod_4}
\begin{gather*}
a_\ell(f) \equiv \begin{cases} 2 \pmod{4} & \mbox{if } \ell \equiv 1,3,5 \pmod{8}, \\ 0 \pmod{4} & \mbox{if } \ell \equiv 2, 7 \pmod{8}. \end{cases}
\end{gather*}
\end{proposition}

The proof of Proposition \ref{prop:f_mod_4} will occupy the rest of this subsection.

As information about $a_{\ell}(\Delta) \pmod{4}$ is available by Kolberg's congruences (cf.~\cite{kolberg}, also~\cite{swinnerton-dyer}), the key to Proposition~\ref{prop:f_mod_4} is information about the coefficients of $\Delta^3 \pmod{2}$.

\begin{lemma}\label{lem:deltamod2} We have
\begin{gather*}
\Delta \equiv \sum_{m=0}^{\infty} q^{(2m+1)^2} \pmod{2} .
\end{gather*}

As a consequence,
\begin{gather*}
a_{n}\big(\Delta^3\big) \equiv q_3(n) \pmod{2}
\end{gather*}
for $n\in\N$ where $q_3(n)$ denotes the number of representations of $n$ as a sum of $3$ odd squares of integers.
\end{lemma}
\begin{proof} The second statement is of course an immediate consequence of the first. For the first statement, consider Dedekind's $\eta$-function. We have $\Delta = \eta^{24}$. There is a classical identity, attributed to Jacobi, between $\eta^3$ and a certain theta-series
\begin{gather*}
\eta^3(q) = \sum_{\substack{m\in\Z \\ m\equiv 1\pmod{4}}} m \cdot q^{\frac{m^2}{8}} ,
\end{gather*}
cf.\ for instance \cite[p.~271]{petersson} with the definition of the theta-series on p.~26. Of course this implies
\begin{gather*}
\eta^3(q) \equiv \sum_{n \geq 0} q^{\frac{(2n+1)^2}{8}} \pmod{2}
\end{gather*}
and the first statement follows by raising to the $8$th power.
\end{proof}

We then need information about the parity of $q_3(n)$. This can be obtained from the following classical theorem, cf.\ \cite[Section~9, Theorem 3]{grosswald} and \cite[Section~3]{onorobinswahl}.

\begin{theorem}[Eisenstein, Dirichlet] For $n\in\N$, let $r_3(n)$ be the number of solutions in $\Z^3$ of $n = x^2 + y^2 + z^2$. Let $\big(\frac{r}{n}\big)$ denote the Jacobi symbol and $[x]$ the integer part of $x$. Then
\begin{gather*}
r_3(n) = \begin{cases}\displaystyle 24 \sum_{d^2 | n}\sum_{r=1}^{[\frac{n}{4}]} \left(\frac{r}{n/d^2} \right) & \mbox{if } n \equiv 1 \pmod{4}, \\
\displaystyle 8 \sum_{d^2 | n}\sum_{r=1}^{[\frac{n}{2}]} \left(\frac{r}{n/d^2} \right) & \mbox{if } n \equiv 3 \pmod{4}.
\end{cases}
\end{gather*}
\end{theorem}

This has the following special case.
\begin{corollary}\label{cor:q_3} Let $\ell \equiv 3 \pmod{8}$ be a prime. Then
\begin{gather*}
q_3(\ell) = R - N,
\end{gather*}
where $R$ (respectively $N$) is the number of quadratic residues $($respectively, quadratic nonresidues$)$ in the interval $\big[1, \ldots, \frac{\ell - 1}{2}\big]$.

Consequently,
\begin{gather*}
q_3(\ell) = 2R - \frac{\ell-1}{2}
\end{gather*}
is odd.
\end{corollary}
\begin{proof} If $\ell = x^2 + y^2 + z^2 \equiv 3 \pmod{8}$, then we must have $x \equiv y \equiv z \equiv 1 \pmod{2}$. Thus $q_3(\ell) = \frac{1}{8} r_3(\ell)$ (here the $\frac{1}{8}$ corrects for allowing signs in the count corresponding to~$r_3$).

The last statement follows by noting that
\begin{gather*}
N = \frac{\ell-1}{2} - R
\end{gather*}
for any odd prime $\ell$.
\end{proof}

\begin{proof}[Proof of Proposition \ref{prop:f_mod_4}] Notice first that $E_4^3 \equiv 1 \pmod{4}$. Let $\ell$ denote a prime number. Since
\begin{gather*}
a_\ell(\Delta) \equiv 1 + \ell^{11} \pmod{4}
\end{gather*}
when $\ell$ is odd, cf.\ \cite{kolberg} (also \cite{swinnerton-dyer}), and since $\tau(2) = -24 \equiv 0 \pmod{4}$, we have
\begin{gather*}
a_\ell\big(E_4^6 \Delta\big) \equiv \begin{cases} 2 \pmod{4} & \mbox{if } \ell \equiv 1 \pmod{4}, \\ 0 \pmod{4}& \mbox{if } \ell \equiv 2,3 \pmod{4} .
\end{cases}
\end{gather*}

Since $a_\ell(f) \equiv a_\ell\big(E_4^6 \Delta\big) + 2q_3(\ell) \pmod{4}$ by Lemma~\ref{lem:deltamod2}, since clearly $q_3(\ell) = 0$ unless $\ell\equiv 3 \pmod{8}$, and since $q_3(\ell)$ is odd when $\ell \equiv 3 \pmod{8}$ by Corollary~\ref{cor:q_3}, the desired formula for~$a_\ell(f) \pmod{4}$ follows.
\end{proof}

\subsection{A mod 4 Galois representation}\label{section:gal_repr} We now construct a mod $4$ Galois representation that is unramified outside $2$ and $\infty$, and such that the trace of $\Frob_\ell$, a Frobenius element at the unramified prime~$\ell$, coincides with~$a_\ell(f) \pmod{4}$ for all odd primes~$\ell$. The representation will factor through a certain dihedral extension of $\Q$.

Here, and in what follows, let $K$ denote the following field:
\begin{gather*}
K := \Q\big( i,\sqrt{2},\sqrt{1+i} \big) .
\end{gather*}

\begin{proposition}\label{prop:D_4_field_K} For the field $K = \Q\big( i,\sqrt{2},\sqrt{1+i} \big)$ we have that $K$ is a Galois extension of $\Q$ with Galois group isomorphic to the dihedral group $D_4$ of order $8$:
\begin{gather*}
\Gal(K/\Q) \cong D_4 = \big\langle r,s \,|\, r^4 = s^2 = 1,\, srs = r^{-1} \big\rangle
\end{gather*}
with the generators $r$, $s$ acting as follows
\begin{gather*}
si = i,\qquad s\sqrt{2} = -\sqrt{2}, \qquad s\sqrt{1+i} = \sqrt{1+i} ,\\
ri = -i, \qquad r\sqrt{2} = -\sqrt{2}, \qquad r\sqrt{1+i} = \frac{1-i}{\sqrt{2}}\cdot \sqrt{1+i}.
\end{gather*}
\end{proposition}

The elementary proof is left to the reader. (For more general information about the structure of Galois extensions with Galois group isomorphic to $D_4$, confer~\cite{ik3} or~\cite[Chapter~2]{jly}.)

\begin{proposition}\label{prop:D_4_mod_4} We have an embedding
\begin{gather*}
D_4 \hookrightarrow \GL_2(\Z/4\Z)
\end{gather*}
given by
\begin{gather*}
r \mapsto = \abcd{1}{2}{1}{1} ,\qquad s \mapsto \abcd{-1}{2}{2}{-1} .
\end{gather*}
\end{proposition}
\begin{proof} One checks that the matrices
\begin{gather*}
X := \abcd{1}{2}{1}{1} ,\qquad Y:= \abcd{-1}{2}{2}{-1}
\end{gather*}
in $\GL_2(\Z/4\Z)$ have orders $4$ and $2$, respectively, and that $YXY = X^{-1}$.
\end{proof}

\begin{definition}\label{def:repr_rho} Denote now and in what follows by $\rho$ the following representation of $G_{\Q}$ into $\GL_2(\Z/4\Z)$:
\begin{gather*}
\rho \colon \ G_{\Q} \twoheadrightarrow \Gal(K/\Q) \cong D_4 \hookrightarrow \GL_2(\Z/4\Z),
\end{gather*}
where the surjection $G_{\Q} \twoheadrightarrow \Gal(K/\Q)$ is the natural surjection, the isomorphism $\Gal(K/\Q) \cong D_4$ is as given above, and the embedding $D_4 \hookrightarrow \GL_2(\Z/4\Z)$ is the one given by Proposition \ref{prop:D_4_mod_4}.
\end{definition}

\begin{proposition}\label{prop:D_4_gal_repr} The representation $\rho$ is unramified outside $2\cdot\infty$, and we have
\begin{gather*}
\Tr \rho(\Frob_\ell) = (a_\ell(f) \pmod{4})
\end{gather*}
for odd primes $\ell$.
\end{proposition}

\begin{proof} Since $K/\Q$ is unramified outside $2\cdot \infty$, so is $\rho$.

Consider the subgroup lattice for $\Gal(K/\Q) \cong D_4$:
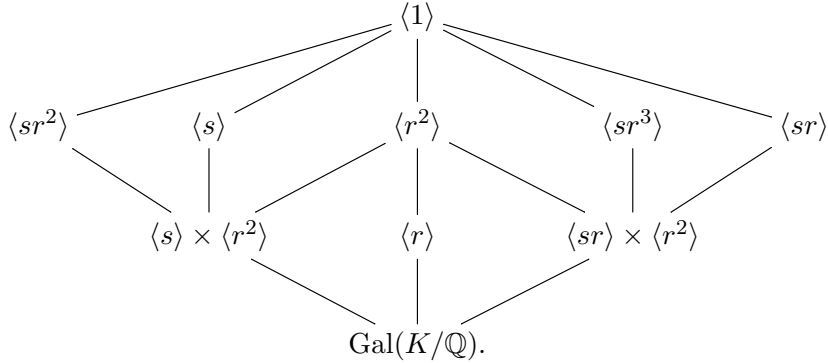
\begin{gather*}
\xymatrix{ & & \langle 1 \ar@{-}[drr]\ar@{-}[dr]\ar@{-}[d]\ar@{-}[dl]\ar@{-}[dll]\rangle & & \\
 \langle s r^2 \ar@{-}[dr] \rangle & \langle s \rangle \ar@{-}[d] & \langle r^2 \rangle \ar@{-}[dl]\ar@{-}[d]\ar@{-}[dr] & \langle s r^3 \rangle \ar@{-}[d]& \langle s r \rangle\ar@{-}[dl] \\
 & \langle s \rangle \times \langle r^2 \rangle \ar@{-}[dr] & \langle r \rangle \ar@{-}[d] & \langle sr \rangle \times \langle r^2 \rangle \ar@{-}[dl] &\\
 & & \Gal(K/\Q). & & }
\end{gather*}

The three lowest-level subgroups $\langle s \rangle \times \langle r^2 \rangle$, $\langle r \rangle$, and $\langle sr \rangle \times \langle r^2 \rangle$ are those of index $2$ and correspond to the quadratic subfields $\Q(i)$, $\Q\big(i\sqrt{2}\big)$, and $\Q\big(\sqrt{2}\big)$, respectively. Thus we find that:
\begin{enumerate}\itemsep=0pt
 \item $\Frob_\ell \in \big\{1, r^2\big\} \Leftrightarrow \mbox{$\ell$ splits in $\Q(i)$, $\Q\big(i\sqrt{2}\big)$, and $\Q\big(\sqrt{2}\big)$}$,
 \item $\Frob_\ell \in \big\{r, r^3\big\} \Leftrightarrow \mbox{$\ell$ splits in $\Q\big(i\sqrt{2}\big)$, but not in $\Q(i)$ or $\Q\big(\sqrt{2}\big)$}$,
 \item $\Frob_\ell \in \big\{s, s r^2\big\} \Leftrightarrow \mbox{$\ell$ splits in $\Q(i)$, but not in $\Q\big(i\sqrt{2}\big)$ or $\Q\big(\sqrt{2}\big)$}$,
 \item $\Frob_\ell \in \big\{sr, s r^3\big\} \Leftrightarrow \mbox{$\ell$ splits in $\Q\big(\sqrt{2}\big)$, but not in $\Q(i)$ or $\Q\big(i\sqrt{2}\big)$}$.
\end{enumerate}

By the laws governing the splitting of primes in quadratic extensions of $\Q$, the cases $(1)$ through $(4)$ correspond to:
\begin{enumerate}\itemsep=0pt
 \item $\legendre{-1}{\ell} = \legendre{2}{\ell} = 1 \Leftrightarrow \ell \equiv 1 \pmod{8}$,
 \item $\legendre{-1}{\ell} = \legendre{2}{\ell} = -1 \Leftrightarrow \ell \equiv 3 \pmod{8}$,
 \item $\mbox{$\legendre{-1}{\ell} = 1$ and $\legendre{2}{\ell} = -1$} \Leftrightarrow \ell \equiv 5 \pmod{8}$,
 \item $\mbox{$\legendre{-1}{\ell} = -1$ and $\legendre{2}{\ell} = 1$} \Leftrightarrow \ell \equiv 7 \pmod{8}$.
\end{enumerate}

The proof is then finished by Proposition \ref{prop:f_mod_4} when we check that, under the embedding $D_4\hookrightarrow \GL_2(\Z/4\Z)$ given by Proposition~\ref{prop:D_4_mod_4}, the elements $1$, $r$, $r^2$, $r^3$, $s$, $sr^2$ are sent to matrices of trace~$2$ whereas the elements~$sr$, $sr^3$ are sent to matrices of trace~$0$.
\end{proof}

\subsection{An elliptic curve and its 4-torsion}\label{sect:ell} We now prepare the proof of Theorem \ref{thm:main} by studying the Galois action on the $4$-torsion of the elliptic curve that it involves.

\begin{proposition}\label{prop:ell_4_torsion} For the elliptic curve
\begin{gather*}
E \colon \ y^2 = x^3+x^2+x+1
\end{gather*}
we have
\begin{gather*}
\rho_{E,4} \cong \rho,
\end{gather*}
where $\rho_{E,4}$ is the representation of $G_{\Q}$ on $E[4]$ and $\rho$ is the representation of the previous section.
\end{proposition}
\begin{proof} The points of order $2$ on $E$ are found to be
\begin{gather*}
Q_1 := (-i,0), \qquad Q_2 := (-1,0), \qquad Q_3 := (i,0).
\end{gather*}

Using standard formulas for halving points, one finds points $P_j$, $j=1,2,3$, with $2\cdot P_j = Q_j$:
\begin{gather*}
x(P_1) = -i - i\sqrt{2}\sqrt{1+i} ,\qquad y(P_1) = -2i - \sqrt{2}\sqrt{1+i} - i\sqrt{2}\sqrt{1+i} ,\\
x(P_2) = -1+\sqrt{2}, \qquad y(P_2) = \sqrt{1+i} + i\sqrt{1+i} - i\sqrt{2} \sqrt{1+i} ,\\
x(P_3) = i - \sqrt{1+i} - i\sqrt{1+i} , \qquad y(P_3) = 2i - 2i\sqrt{1+i} .
\end{gather*}

We see that $\Q(E[4])$ is the $D_4$-extension $K/\Q$ above. We have that $P_1$, $P_2$ is a basis for $E[4]$, and one verifies that
\begin{gather*}
P_1+P_2=P_3.
\end{gather*}

Identifying $\Gal(K/\Q) \cong D_4 = \langle r,s\rangle$ as in Definition \ref{prop:D_4_field_K}, one computes
\begin{gather*}
r P_1 = P_3 = P_1 + P_2,\\
r P_2 = 2\cdot P_1 + P_2 ,\\
sP_1 = P_1 + 2\cdot P_3 = -P_1 + 2\cdot P_2 ,\\
sP_2 = P_2 + 2\cdot P_3 = 2\cdot P_1 - P_2 ,
\end{gather*}
whence the claim follows by the definition of $\rho$, Definition \ref{def:repr_rho}. (Notice that we have chosen to write the Galois action from the left. This means that the representing matrices should be considered as acting on column vectors of coordinates w.r.t.\ $(P_1,P_2)$.)
\end{proof}

\begin{proof}[Proof of Theorem \ref{thm:main}] That $\Tr \rho_2(\Frob_\ell) \equiv a_\ell(f) \pmod{4}$ for odd primes $\ell$ follows by combining Propositions~\ref{prop:D_4_gal_repr} and~\ref{prop:ell_4_torsion}. Thus, $a_\ell(f) \equiv a_\ell(g) \pmod{4}$ for odd primes $\ell$ where $g$ is the cusp form of weight $2$ and level $128$ attached to $E$. Since $g$ is an eigenform, since $f$ is an eigenform modulo $4$, since $a_2(f) \equiv 0 \pmod{4}$ (cf.\ Proposition~\ref{prop:f_mod_4}), and since $a_2(g)=0$, cf.\ \cite[\href{http://www.lmfdb.org/EllipticCurve/Q/128/a/2}{Elliptic Curve 128.a2}]{lmfdb}, we deduce that
\begin{gather*}
f \equiv g \pmod{4}.\tag*{\qed}
\end{gather*}\renewcommand{\qed}{}
\end{proof}

\section{Further remarks}\label{sect:remarks}

{\bf 1.} There are $4$ newforms of weight $2$ on $\Gamma_0(128)$, all with coefficients in $\Z$, i.e., corresponding to elliptic curves over $\Q$ of conductor $128$. The forms are all congruent to each other modulo $4$ so we could have used anyone of them (and the corresponding elliptic curve) in Theorem~\ref{thm:main}.

{\bf 2.} Our initial purpose of studying the form $f$ above was, as described in Section~\ref{section:background}, to choose a~form that really appeared hard to lift to a true, characteristic $0$ eigenform and then test Question~\ref{q:strong_lifts} against it. Our main result then says that this form, that seemed hard to lift, in fact does not provide us with an affirmative answer to that question.

Of course, once we have that result, we also know that $f$ can be lifted in infinitely many ways in the sense of Question~\ref{q:strong_lifts_outside_S} since we can twist $f$ (and the lifting form) with a quadratic character without changing the Fourier coefficients modulo~$4$ (as these are just~$0$ and~$2$ modulo~$4$, cf.\ Proposition~\ref{prop:f_mod_4}.)

{\bf 3.} One might wonder, if we just wanted to show that our form $f \pmod{4}$ lifts to some eigenform in characteristic $0$, could we not work with the weight $1$ eigenform that is attached to a complex representation of $\Gal(K/\Q)$ with~$K$ the dihedral field of Section~\ref{section:gal_repr}? The answer is no: if~$h$ is such a form of weight~$1$ one checks that we do not have
\begin{gather*}
a_\ell(f) \equiv a_\ell(h) \pmod{4}
\end{gather*}
for almost all primes $\ell$.

{\bf 4.} S.~Deo drew our attention to the fact that one can construct another Galois representation with the properties of Proposition~\ref{prop:D_4_gal_repr} by working with the field $L = \Q\big( i,\sqrt[4]{2}\big)$ instead of the field~$K$ of Section~\ref{section:gal_repr}: $L/\Q$ is again a Galois extension with Galois group isomorphic to the dihedral group~$D_4$ of order~$8$:
\begin{gather*}
\Gal(L/\Q) \cong D_4 = \big\langle r,s \,|\, r^4 = s^2 = 1,\, srs = r^{-1} \big\rangle
\end{gather*}
with the generators $r$, $s$ acting as follows
\begin{gather*}
si = -i,\qquad s\sqrt[4]{2} = \sqrt[4]{2}, \qquad ri = i, \qquad r\sqrt[4]{2} = i \sqrt[4]{2}.
\end{gather*}

If one considers the embedding $D_4 \hookrightarrow \GL_2(\Z/4\Z)$ given by
\begin{gather*}
r \mapsto \abcd{1}{2}{1}{1} ,\qquad s \mapsto \abcd{1}{0}{1}{-1}
\end{gather*}
rather than the one of Proposition \ref{prop:D_4_mod_4}, then for the resulting Galois representation
\begin{gather*}
\rho' \colon \ G_{\Q} \twoheadrightarrow \Gal(L/\Q) \cong D_4 \hookrightarrow \GL_2(\Z/4\Z)
\end{gather*}
one finds by analysis similar to the proof of Proposition \ref{prop:D_4_gal_repr} that{\samepage
\begin{gather*}
\Tr \rho'(\Frob_\ell) = (a_\ell(f) \pmod{4})
\end{gather*}
for odd primes $\ell$.}

We do not know whether this representation is strongly modular in the sense of being the mod $4$ reduction of the $2$-adic Galois representation attached to some newform. But we can say that it is not isomorphic to the Galois representation of the $4$-torsion of an elliptic curve defined over $\Q$ (that is, if it is strongly modular, the reason is different than the one we found for the representation $\rho$ of Section~\ref{section:gal_repr}): Suppose that we had an elliptic curve~$C$ defined over~$\Q$ such that the representation of~$G_{\Q}$ on $C[4]$ were isomorphic to $\rho'$ above. In particular, $\Q(C[4]) = L = \Q\big( i,\sqrt[4]{2}\big)$. Consider a prime $\ell$ that is $\equiv 3\pmod{8}$ and sufficiently large so as not to divide the conductor of~$C$. Now, in the field $L$ such a prime has inertial degree $2$, i.e., the residue field will be $\F_{\ell^2}$. Since $C$ has good reduction at $\ell$, we find that $C[4]$ embeds into $C(\F_{\ell^2})$ and so we must have $\# C(\F_{\ell^2}) \equiv 0 \pmod{16}$. Put $a:=\ell+1 - \# C(\F_{\ell})$ and let $\alpha$, $\beta$ be the roots of $x^2-ax+\ell$. Then
\begin{gather*}
(\ell+1)^2 - a^2 = \ell^2 + 1 - (\alpha+\beta)^2 + 2\alpha\beta = \ell^2 + 1 - \alpha^2 - \beta^2 = \# C(\F_{\ell^2}) \equiv 0 \pmod{16},
\end{gather*}
whence $a\equiv 0 \pmod{4}$ since $\ell\equiv 3 \pmod{8}$. But $a = a_\ell(h)$ if $h$ is the newform attached to $C$. Thus,
\begin{gather*}
a_\ell(f) \equiv 2 \not\equiv 0 \equiv a_\ell(h) \pmod{4}
\end{gather*}
for all sufficiently large primes $\ell \equiv 3 \pmod{8}$.

{\bf 5.} Theorem 2 of \cite{ckw} attaches a Galois representation to a dc-weak eigenform mod $p^m$ provided that the Galois representation attached to the form mod $p$ is absolutely irreducible. ``Attached'' means here that the traces of Frobenii are almost all given by the corresponding Hecke eigenvalues. The same can be proven under an assumption of ``residually multiplicity-freeness'', as well as the Vandiver conjecture, by combining results of \cite{bellaiche_khare} and \cite[Proposition~1.6.1]{bellaiche_chenevier}. This, however, would still not apply to our example of $(f \pmod{4})$ (here, the semi-simplification of the attached residual representation is just the sum of two trivial characters.) Hence one can remark that in the case of $(f \pmod{4})$, we still have an attached Galois representation, cf.\ Proposition~\ref{prop:D_4_gal_repr}. Moreover, and as could be anticipated from \cite[Th\'{e}or\`{e}me~1]{carayol} (which is the real background for \cite[Theorem~2]{ckw}), this Galois representation is not uniquely determined by being ``attached'' to $(f \pmod{4})$ in the above sense, as the previous remark shows.

However, we must emphasize that the Galois representations are not our primary focus in the present paper: they are merely tools that allow us to identity newforms congruent to \smash{$(f \pmod{4})$}. Those congruences mod $4$ between $f$ and true characteristic $0$ eigenforms are the main focus of the paper, as noted in the motivating Section~\ref{section:background}.

{\bf 6.} We would like to note that Theorem \ref{thm:main} can be seen ``the other way around'', that is, as producing $f\pmod{4}$ given the form $g$ of level $128$. I.e., as an example of level-lowering modulo~$4$. As already noted in Section~\ref{section:background} above, it is known, at least when $p\ge 5$, that one can ``strip powers of $p$ away from the level'' of a strong eigenform modulo $p^m$ (cf.\ \cite[Proposition~19]{ckw} which is based on work of Hida.) However, the result of this kind of level-lowering seems naturally to be a~dc-weak eigenform at the lowered level (though, it seems, this can be strengthened by utilizing \cite[Theorem~4.3]{nadim_level1}.)

The thesis \cite{panos_thesis} contains results on level-lowering modulo higher prime powers (none of the results apply to the situation of Theorem~\ref{thm:main}), and the shape of these statements is that, after lowering the level, one ends up with a weak eigenform even if one started with a strong eigenform modulo~$p^m$. Theorem~\ref{thm:main} provides an example where the level-lowered eigenform is weak, but provably not strong at that level (which in this case is~$1$.) It is probably the first example of this kind.

{\bf 7.} The second author has determined all weak eigenforms on $\SL_2(\Z)$ with coefficients in~$\Z/4\Z$. There are $8$ in all, and, furthermore, they all lift to characteristic $0$ eigenforms of some $2$-power level and some weight (here, ``lift'' means that coefficients prime to some fixed $M$ are reductions mod~$4$ of the corresponding coefficients of the higher-level form.) One can determine these weak eigenforms by using the weight bounds that come from Nicolas--Serre theory, cf.\ \cite[Section~3.2]{krw}.

\subsection*{Acknowledgements} The authors would like to thank Shaunak Deo and Gabor Wiese for interesting discussions relating to this paper as well as to other questions concerning modular forms modulo prime powers. We thank Ariel Pacetti for comments on the first draft of the paper. We also thank the anonymous referees for comments and suggestions that helped improve the exposition.
The second author was supported by a Postdoctoral Fellowship at the National Center for Theoretical Sciences, Taipei, Taiwan.
The first author would like to thank Noriko Yui for good contact, collaboration, and interesting exchange over many years.

\pdfbookmark[1]{References}{ref}
\LastPageEnding

\end{document}